\documentclass[11pt]{amsart}   	
\pdfoutput=1

\usepackage[margin=1in]{geometry}

\usepackage{amsmath,amssymb,amsfonts,amsthm}
\usepackage{hyperref}
\usepackage{mathtools}
\usepackage{tikz-cd}
\usepackage{tikz}
\usepackage{enumitem}
\usepackage{makecell}
\setcellgapes{4pt}

\newcommand{\add}{\mathsf{add}}

\newcommand{\Fac}{\mathsf{Fac}}

\DeclareMathOperator{\Hom}{\mathrm{Hom}}

\newtheorem{theorem}{Theorem}[section]
\newtheorem{corollary}[theorem]{Corollary}
\newtheorem{lemma}[theorem]{Lemma}
\newtheorem{proposition}[theorem]{Proposition}

\theoremstyle{definition}
\newtheorem{definition}[theorem]{Definition}
\newtheorem{example}[theorem]{Example}

\usepackage[style=alphabetic]{biblatex}

\usepackage[margin=1in]{geometry}

\title{Reduction techniques to identify connected components of mutation quivers}
\author{Håvard Utne Terland}
\address{Department of Mathematical Sciences, Norwegian University of Science and Technology (NTNU), 7491 Trondheim, NORWAY}
\email{havard.u.terland@ntnu.no}

\begin{document}

\maketitle


\begin{abstract}
   Important objects of study in $\tau$-tilting theory include the $\tau$-tilting pairs over an algebra on the form $kQ/I$, with $kQ$ being a path algebra and $I$ an admissible ideal.
   
   In this paper, we study aspects of the combinatorics of mutation quivers of support $\tau$-tilting pairs, simply called mutation quivers. In particular, we are interested in identifying connected components of the underlying graphs of such quivers.
   
   We give a class of algebras with two simple modules such that every algebra in the class has at most two connected components in its mutation quiver, generalizing a result by Demonet, Iyama and Jasso \cite[Theorem 6.7]{dij17}. We also give examples of algebras with strictly more than two components in their mutation quivers.
\end{abstract}




\section{Introduction}
$\tau$-tilting theory was introduced in \cite{tau} as a generalization of classical tilting theory. An important feature of $\tau$-tilting theory is that mutation of support $\tau$-tilting pairs is always possible, inducing a quiver structure on the set of support $\tau$-tilting pairs where arrows indicate (left) mutations and whose underlying graph is regular.

It is well-known that algebras with only finitely many $\tau$-rigid objects up to isomorphism (\textit{$\tau$-tilting finite algebras)} have connected mutation quivers, a property which need not hold for algebras with infinitely many $\tau$-rigid modules up to isomorphism (\textit{$\tau$-tilting infinite algebras}). This paper deals with the problem of identifying connected components of such mutation quivers for various algebras.

From a computational perspective, we consider this a natural problem. Indeed, as mutation can be computed algorithmically, mutation quivers may be traversed. Thus if one can determine all components of a mutation quiver, one can also effectively enumerate all support $\tau$-tilting pairs up to isomorphism using induction, starting from a representative of each component. 

Inspired by Demonet, Iyama and Jasso, who in \cite[Theorem 6.17]{dij17} prove that a certain algebra with two simple modules has exactly two connected components in its mutation quiver, we continue the search for algebras where we can determine exactly how many components its mutation quiver has. In particular, we develop in Section \ref{technique_section} a reduction technique allowing us to in some cases transport 2-term complexes between module categories while maintaining rigidity. Using this technique, we generalize \cite[Theorem 6.17]{dij17}.

We then show that not all algebras have at most two components in their mutation quiver. We give an elementary proof that there is a disconnected algebra with exactly 4 components in its mutation quiver, and show using the techniques developed by Brüstle, Smith and Treffinger in \cite{Br_stle_2019} that we can construct a connected algebra one with at least 4 connected components in its mutation quiver.


\section*{Acknowledgments}
This work was financed by the Norwegian Research Council, through the project Applications of Reduction Techniques and Computations in Representation Theory (ARTaC) with grant number FRINATEK 301375. The author is grateful for the advice of Aslak Bakke Buan and Eric J. Hanson.

\section{Setting}
All algebras in this paper are assumed to be on the form $kQ/I$, where $k$ is a field which need not be algebraically closed, $Q$ a finite quiver and $I$ an admissible ideal of the path algebra $kQ$. In particular, all algebras considered are finite dimensional.

Given a path algebra $kQ$, we denote by $r$ the ideal generated by all paths of length $1$. For an algebra $A = kQ/I$, we denote by $\underline{r} \subseteq A$ the ideal induced by $r \subseteq kQ$ via the canonical epimorphism $kQ \to kQ/I$.

We denote by $\text{mod } A$ the category of finite-dimensional left $A$-modules. We implicitly identify modules over $A$ with their corresponding representations over the quiver of $A$.

Given a module $M$ over an algebra $kQ/I$, we denote by $[M]$ the dimension vector of $M$. This demands an implicit order on the points or vertices of $Q$. The simple module in point $i$ is denoted $S(i)$ and has dimension vector $e_i$, where $e_i$ is the $i$'th element of the standard basis for $\mathbb{Z}^n$.

An $A$-module is called basic if no two non-zero indecomposable summands are isomorphic. 

Given $A$-modules $M$ and $N$, we denote by $hom_A(M,N)$ the dimension of $\Hom_A(M,N)$ as a $k$-vector space.

Given a module $M$, we denote by $\add M$ the category of modules $N$ such that all indecomposable summands of $N$ are summands of $M$. For such a module $M$, we denote by $\Fac M$ the subcategory of modules which can be exhibited as a quotients of modules in $\add M$.

\subsection{$\tau$-tilting theory}
We recall some important definitions from \cite{tau}. An $A$-module $M$ is called $\tau$-rigid if $\Hom_A(M,\tau M) = 0$.

\begin{definition}\cite[Definition 0.3]{tau}
	A pair $(M,P)$ where $M$ is $\tau$-rigid and $P$ is projective such that $\Hom_A(P,M) = 0$ is called a $\tau$-rigid pair. 
	
	Except when otherwise is stated, it is assumed that both elements $M$ and $P$  of such a pair $(M,P)$ are basic modules. 
	
\end{definition}

We say that a $\tau$-rigid pair $(N,Q)$ is a summand of a $\tau$-rigid pair $(M,P)$ if $N$ is a summand of $M$ and $Q$ is a summand of $P$.

A $\tau$-rigid pair is called support $\tau$-tilting if $|M| + |P| = |A|$ and almost support $\tau$-tilting if $|M|+|P| = |A|-1$. By \cite[Proposition 1.2, Proposition 2.3]{tau}, a $\tau$-rigid pair $(M,P)$ has at most $|A|$ summands, that is, $|M| + |P| \leq |A|$.

An important property of support $\tau$-tilting pairs is that mutation is always possible.

\begin{theorem}\cite[Theorem 2.18]{tau}
	For an almost complete support $\tau$-tilting pair $(M,P)$, there are exactly two distinct support $\tau$-tilting pairs $(M_0,P_0)$ and $(M_1,P_1)$ having $(M,P)$ as a summand.
\end{theorem}

In the above theorem, $(M_0,P_0)$ and $(M_1,P_1)$ are said to be mutations of each other. We may denote this as $(M_1,P_1) = \mu_X(M_0,P_0)$ where $X$ is the unique indecomposable summand of $(M_0,P_0)$ which is not a summand of $(M_1,P_1)$.

Mutation induces a natural graph structure on the set of support $\tau$-tilting pairs. In fact, one can also naturally give mutation a direction.

\begin{proposition}\cite[Definition-Proposition 2.28, Theorem 2.7]{tau}
	Let $T_1 = (M_1,P_1)$ and $T_2 = (M_2,P_2)$ be support $\tau$-tilting pairs which are mutations of each other. Then either $\text{Fac }M_1 \subset \text{Fac }M_2$ or $\text{Fac }M_2 \subset \text{Fac }M_1$. In the first case, $T_2$ is called a right-mutation of $T_1$ and in the second case $T_2$ is called a left-mutation of $T_1$. Note that if $T_2$ is a left-mutation of $T_1$ if and only if $T_1$ is a right-mutation of $T_2$.
\end{proposition}

It is important to note that left mutations may be explicitly computed, utilizing left approximations. Right mutations may be computed by first passing to the opposite algebra, computing the appropriate left mutation, and going back to the original algebra. These ideas and methods are developed in \cite{tau}.

We now recall the definition of mutation quivers.
\begin{definition}\cite[Definition 2.29]{tau}
	For an algebra $A$, we define the mutation quiver $Q(s\tau\text{-tilt } A)$ to be the quiver with vertex set being the set of support $\tau$-tilting pairs up to isomorphism and arrows $T_1 \to T_2$ exactly when $T_2$ is a left mutation of $T_1$.
\end{definition}

It is not surprising that the mutation quiver of a disconnected algebra $A \times B$ is completely determined by the mutation quiver of $A$ and the mutation quiver of $B$. We now demonstrate how this is done precisely.

\begin{definition}
	Let $Q_1,Q_2$ be quivers with vertex sets $Q_1^0$, $Q_2^0$ and arrow sets $Q_1^1$, $Q_2^1$. The Cartesian product quiver $Q = Q_1 \mathbin{\Box} Q_2$ is defined as the quiver with vertex set $Q_1^0 \times Q_2^0$ and edges on the form $(x_1,y_1) \to (x_2,y_2)$ such that either $x_1 = x_2$ and $y_1 \to y_2$ is an edge in $Q_2^1$, or $y_1 = y_2$ and $x_1 \to x_2$ is an edge in $Q_1^2$.
\end{definition}

	The proof of the below lemma follows immediately from classical results on the Auslander-Reiten transform on disconnected algebras and the definition of mutation. A complete proof may be found in the master thesis of the author.

\begin{lemma}\label{disconnected_mutation_quiver}
	Let $\Lambda = A \times B$ be a disconnected algebra. Then \[Q(s\tau\text{-tilt} \Lambda) \cong Q(s\tau\text{-tilt} A)\mathbin{\Box}Q(s\tau\text{-tilt} B)\] as quivers.
\end{lemma}

Looking at vectors coming from projective presentations of $\tau$-rigid modules gives a combinatorial tool crucial to $\tau$-tilting theory.

\begin{definition}\cite[Section 5]{tau}
	For a module $M$, let $P_1 \to P_0 \to M \to 0$ be a minimal projective presentation of $M$, where $P_1 \cong \bigoplus_{i = 1}^n P(i)^{v_i}$ and $P_0 \cong \bigoplus_{i = 1}^n P(i)^{u_i}$. The vectors $v = (v_1,v_2,\dots,v_n)$ and $u = (u_1,u_2,\dots,u_n)$ then encodes $P_1$ and $P_2$ respectively. $g^M$ is defined to be $u - v$. 
\end{definition}

By \cite[Theorem 5.1]{tau}, $g$-vectors uniquely identify $\tau$-rigid modules up to isomorphism. A $\tau$-rigid pair $(M,P)$ is defined to have $g$-vector $g^M - g^P$. For an ordered support $\tau$-tilting pair $T = (M,P)$, we define its $G$-matrix as the matrix whose $i$'th column is the $g$-vector of the $i$'th summand of $T$. Such $G$-matrices are by \cite[Theorem 5.1]{tau} invertible integer matrices.


\subsection{2-term pre-silting complexes}
Instead of working with $\tau$-rigid pairs, it is sometimes convenient to consider the equivalent notion of 2-term pre-silting complexes. We here introduce the results needed in this paper, and refer the reader to \cite[Section 3]{tau} for a more detailed treatment on the topic.


We use co-homological notation on complexes. A complex $\mathbb{P}$ is called $2$-term if $\mathbb{P}^i = 0$ for all $i \notin\{0,-1\}$. 

\begin{definition}
	Fix an algebra $A$. A $2$-term object $\mathbb{P}$ in the category $K^b(\add A)$ is called presilting if it is rigid, that is $\Hom_{K^b(\add A)}(\mathbb{P},\mathbb{P}[1]) = 0$
\end{definition} 

We can now state the connection between $\tau$-rigid pairs and 2-term pre-silting complexes.

\begin{proposition}\cite[Proposition 3.7]{tau}
	Let $P_1 \xrightarrow{d_1} P_0 \xrightarrow{d_0} M$ be a minimal projective presentation of a $\tau$-rigid module $M$. Then we have
	
	\begin{enumerate}
		\item $P_1 \xrightarrow{d_1} P_0$ is a pre-silting complex.
		\item If $(M,Q)$ is a $\tau$-rigid pair, then $P_1 \oplus Q \xrightarrow{\begin{bmatrix}d_0 & 0\end{bmatrix}} P_0$ is a 2-term presilting complex.
	\end{enumerate}
\end{proposition}

Pre-silting complexes coming from support $\tau$-tilting objects are by \cite[Theorem 3.2]{tau} exactly the 2-term silting complexes.


By abuse of notation, we denote by $g^\mathbb{P}$ the $g$-vector of the $\tau$-rigid pair corresponding to a 2-term pre-silting complex $\mathbb{P}$.

\subsection{Wall and chamber structures}
Considering $g$-vectors as geometrical objects in $\mathbb{R}^n$ gives a natural starting point for studying what are called wall and chamber structures for algebras. We endow $\mathbb{R}^n$ with its Euclidean topology and let $\langle -,-\rangle$ denote the standard inner product. 


\begin{definition}
	Let $T = (M,P)$ be a $\tau$-rigid pair. Following \cite{dij17}, the \textit{cone} $C(T)$ induced by $T$ is defined as \[C(T) = \{\sum_{i = 1}^{j} \alpha_ig^{M_i} - \sum_{i = 1}^{k}\beta_ig^{P_i} : \alpha_i,\beta_i \in \mathbb{R}^+ \cup \{0\} \text{ for all } i\}\]
	
	The interior $C^\circ(T) \subset C(T)$ is found by requiring all coefficients $\alpha_i,\beta_i$ to be strictly positive.
\end{definition} 

The wall and chamber structure of an algebra is introduced in \cite{Br_stle_2019} where Brüstle, Smith and Treffinger combine ideas from \cite{bridgeland2017} and \cite{dij17}. Fix an algebra $A$ with $n$ simple modules. For a vector $v \in \mathbb{R}^n$, a module $M$ is called $v$-stable if $\langle v,[M]\rangle = 0$ and $\langle v,[N]\rangle < 0$ for all proper submodules $N$ of $M$.

The stability space of a given module is then defined as the vectors $v$ making the module $v$-stable. Such a stability space is closed under positive scaling, and is called a wall if its linear span has codimension $1$. A chamber is a connected component of $\mathbb{R}^n$ minus all such walls. By \cite[Proposition 3.15]{Br_stle_2019} the interior of a cone induced by a support $\tau$-tilting pair is a chamber. By \cite[Theorem 3.17]{asai2020wallchamber} all chambers are on this form.

The following proposition shows that the cones of support $\tau$-tilting objects intersect in a very controlled fashion. 

\begin{proposition}\cite[Corollary 6.7, b)]{dij17}
	Let $T_1,T_2$ be support $\tau$-tilting pairs. Then $C(T_1) \cap C(T_2) = C(X)$ with $X$ being the maximal $\tau$-rigid pair such that $X$ is a summand of $T_1$ and of $T_2$.
\end{proposition}

These results together suggest that the wall and chamber structure of an algebra gives a geometric model for the mutation quiver of an algebra. 

\cite[Definition 4.1]{Br_stle_2019} formalizes this idea by defining $\mathcal{D}$-generic paths as continuous paths in $\mathbb{R}^n$ moving between chambers passing only one wall at a time and doing so transversally. They proceed to demonstrate in \cite[Theorem 4.3]{Br_stle_2019} that these $\mathcal{D}$-generic paths correspond to mutations in the sense that a sequence of mutations between support $\tau$-tilting pairs induce a $\mathcal{D}$-generic path between the corresponding chambers, passing exactly one wall for each mutation. 




\section{Main results}

We begin this section by studying the techniques we need to prove Theorem \ref{k2-reduction}. We then give some corollaries to this result and lastly discuss algebras with more than two connected components in their mutation quivers.

We remark that both the results and proofs from this section closely follow the master thesis of the author, \cite{terland-master}.

\subsection{A reduction technique for 2-term pre-silting complexes}\label{technique_section}

Before giving the details on our reduction method, we motivate it by the following example.

\begin{example}
	
Consider the algebra $A = kQ/r^3$, where $Q$ is the quiver below.
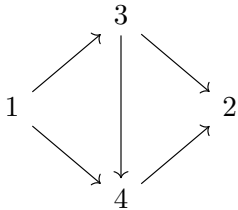
\[\begin{tikzcd}
	& 3 \\
	1 && 2 \\
	& 4
	\arrow[from=2-1, to=1-2]
	\arrow[from=2-1, to=3-2]
	\arrow[from=3-2, to=2-3]
	\arrow[from=1-2, to=2-3]
	\arrow[from=1-2, to=3-2]
\end{tikzcd}\]

There is an obvious equivalence between the additive categories $\add(P(1) \oplus P(2))$ and $\add K_2$, where we by $K_2$ mean the Kronecker algebra.

Since $\tau$-tilting theory can be studied using $2$-term pre-silting complexes, we can transport $\tau$-rigid modules from $K_2$ to $\tau$-rigid modules of $A$ via this equivalence. In fact, for any support $\tau$-tilting pair $T = (M,P)$ of $K_2$ with $g$-vector $v = (v_1,v_2)$, the equivalence between $\add(P(1) \oplus P(2))$ and $\add K_2$ allows us to conclude that there is a $\tau$-rigid pair $T_P = (M_A,P_A)$ over $A$ with $g$-vector $(v_1,v_2,0,0)$. Also, $T$ and $T_P$ both has the same number of indecomposable summands. Thus we can use the $\tau$-tilting theory of $K_2$ to better understand the $\tau$-tilting theory of $A$.

Consider now instead the algebra $B = kQ'/r^3$ with $Q'$ as drawn below.
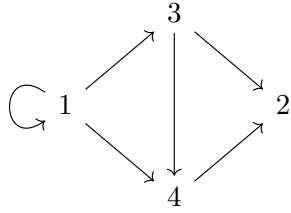
\[
\begin{tikzcd}
	& 3 \arrow[rd] \arrow[dd] &   \\
	1 \arrow[ru] \arrow[rd] \arrow[loop, distance=2em, in=215, out=145] &                         & 2 \\
	& 4 \arrow[ru]            &  
\end{tikzcd}\]

Now, there is no equivalence between $\add(P_B(1) \oplus P_B(2))$ and $\add K_2$, as we have non-trivial endomorphisms of $P_B(1)$. Although the situation is not nearly as nice as when this equivalence hold, we later show that the reduction technique in Lemma \ref{inherit-tech} allows us to transport some $2$-term pre-silting complexes over $K_2$ to $2$-term pre-silting complexes over $B$, meaning that the $\tau$-tilting theory of $K_2$ still can give valuable information about the $\tau$-tilting theory of $B$:

\end{example}

We now present the reduction technique. Fix two algebras $\Lambda$ and $\Gamma$. Fix a pair of basic projective $\Lambda$-modules, not necessarily indecomposable, $(P_\Lambda,Q_\Lambda)$ and a pair of projective $\Gamma$-modules, not necessarily indecomposable, $(P_\Gamma,Q_\Gamma)$.

We are interested in the class $\mathcal{C}$ of 2-term pre-silting complexes on the form \[P' \rightarrow Q'\] where $P'$ lies in $\add P_\Lambda$ and $Q'$ lies in $\add Q_\Lambda$. Assume that there exists an additive functor \[F:\add (P_\Lambda \oplus Q_\Lambda) \to \add(P_\Gamma \oplus Q_\Gamma)\] with the following properties. 
	
\begin{enumerate}
	\item[(F1)] $F(P_\Lambda) = P_\Gamma$
	\item[(F2)] $F(Q_\Lambda) = Q_\Gamma$
	\item[(F3)] $F$ restricted to $\Hom_\Lambda(P_\Lambda,Q_\Lambda)$, considered as a function, is surjective.
	\item[(F4)] For any $P$ in $\add (P_\Lambda \oplus Q_\Lambda)$, we require $g^{F(P)}$ to equal $g^P$ in all coordinates where they are both defined and be $0$ everywhere else, meaning that $F$ essentially respects $g$-vectors.
\end{enumerate}

Concerning the last requirement above, note that $\Gamma$ and $\Lambda$ need not have the same number of simple modules. Thus demanding the two $g$-vectors to be equal would be too strict. We can now state and easily prove the reduction technique.

\begin{lemma}\label{inherit-tech}
	With the setup described above, let $\mathbb{P} = P' \xrightarrow{\phi} Q'$ be a 2-term pre-silting complex over $\Lambda$ lying in $\mathcal{C}$. Then \[F(\mathbb{P}): F(P') \xrightarrow{F(\phi)} F(Q')\] is a $2$-term pre-silting complex over $\Gamma$. 
	
	
\end{lemma}

\begin{proof}
	It is sufficient to prove that \[F(\mathbb{P}) = F(P') \xrightarrow{F(\phi)} F(Q')\] is rigid as a complex, considered as an object in $K^b(\add \Gamma)$. Thus we seek for any $h:F(P') \to F(Q')$ two maps $g_2$ and $g_1$ such that $F(\phi) \circ g_1 + g_2 \circ F(\phi) = h$, as in the diagram drawn below. 
	\[
	\begin{tikzcd}
		& F(P') \arrow[r, "F(\phi)"] \arrow[d, "h"] \arrow[ld, "g_1", dotted] & F(Q') \arrow[ld, "g_2"',dotted] \\
		F(P') \arrow[r, "F(\phi)"] & F(Q')                                                               &                         
	\end{tikzcd}
	\]
		
	By the F3, there exists a map $h'$ such that $F(h') = h$. Since $\mathbb{P}$ is rigid, there exists $g_1'$ and $g_2'$ such that $\phi \circ g_1' + g_2' \circ \phi = h'$. Letting $g_1 = F(g_1')$ and $g_2 = F(g_2')$, we see by functoriality that $F(\phi) \circ g_1 + g_2 \circ F(\phi) = h$ as wanted. 

\end{proof}

When applying the above lemma, some care must be taken. First, $H^0(F(\mathbb{P}))$ need not be indecomposable, even if $\mathbb{P}$ is and we will give an example of this later. Also, although we are not aware of any counterexamples, it can not be concluded that $F(\mathbb{P})$ is indecomposable even if $H^0(F(\mathbb{P}))$ is. In this case, the $g$-vector of $H^0(F(\mathbb{P}))$ cannot be naively computed by inspecting $\mathbb{P}$. Note however that if $\mathbb{P}$ is indecomposable, then $H^0(\mathbb{P})$ is indecomposable as well and its $g$-vector must be $g^\mathbb{P}$. 

We now apply Lemma \ref{inherit-tech} to the case where $\Lambda$ is the Kronecker algebra.

\begin{example}\label{K2example}
	
	Let $\Lambda = K_2$ in Lemma \ref{inherit-tech}, where $K_2$ is the Kronecker algebra, i.e the path algebra drawn below. Let $\Gamma$ be an algebra with ordered projectives $Q(i)$ where $i \in [1,\dots,n]$ and a pair of indecomposable projectives $(Q(2),Q(1))$ such that $hom_\Gamma(Q(2),Q(1)) = 2$. 
	
\[K_2 = k(\begin{tikzcd}
	1 
	\arrow[r, shift left = 0.7ex]
	\arrow[r, shift right = 0.7ex]
	& 2
\end{tikzcd})
\]

	We may then define a functor $F:\add K_2 \to  \add( Q(1) \oplus Q(2))$ by sending the $K_2$-module $P(1)$ to $Q(1)$ and  the $K_2$-module $P(2)$ to $Q(2)$. It is clear that morphisms in $\add K_2$ can be transported functorially to morphisms in $\add (Q(1) \oplus Q(2))$. We see that the requirements F1-F4 are satisfied by $F$.

	We now discuss the consequences of Lemma \ref{inherit-tech} in this scenario. Consider the family $\mathcal{C}$ of $\tau$-rigid modules in $\text{mod }K_2$ with projective presentations on the form $P(2)' \to P(1)'$ where $P(i)'$ lies in $\add P(i)$. These have $g$-vectors on the form $(i+1,-i)$ for integers $i \geq 0$. A finite portion of these vectors are drawn bold in the figure below, and a subset of the other $g$-vectors are drawn with dotted lines. Lemma \ref{inherit-tech} then allows us to conclude that there is a family of 2-term pre-silting objects over $\Gamma$ corresponding to $\tau$-rigid pairs with $g$-vectors on the form $(i+1,-i,0,\dots,0)$.
	
		\begin{figure}[h]
		\begin{center}
			\begin{tikzpicture}[>=stealth,scale=1,line cap=round,
				bullet/.style={circle,inner sep=1.5pt,fill}]
				\foreach \X [count=\Y] in {(1,0),(2,-1),(3,-2),(4,-3)} 
				{\path  \X node(n\Y)[label=right:{$$}]{};
					\draw[thick,->,opacity = 1]  (0,0) -- (n\Y);}
				
				\foreach \X [count=\Y] in {(0,1),(1,-2),(2,-3),(3,-4),(-1,0),(0,-1),(-1,1)} 
				{\path  \X node(n\Y)[label=right:{$$}]{};
					\draw[dashed,->,opacity = 1]  (0,0) -- (n\Y);}
				
			\end{tikzpicture}
		\end{center}
		\caption{A finite portion of the $g$-vectors of $K_2$. }
		\label{k2-wall-and-chamber}
	\end{figure}
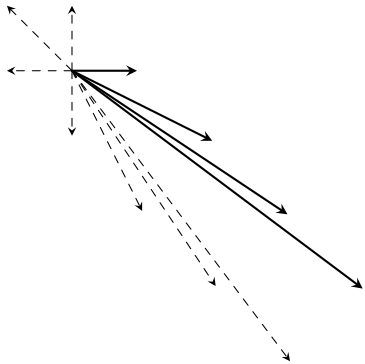
\end{example}

In general, the technique described above need not take indecomposable $\tau$-rigid modules to indecomposable $\tau$-rigid modules, as seen in the example below.

\begin{example}\label{example-decomposing-modules}
Let $A = kQ/\beta^2$ where $Q$ is the quiver below.

\[
\begin{tikzcd}
	1 & 2
	\arrow[from=1-1,to=1-2,"\alpha"]
	\arrow[loop,from = 1-2,to=1-2,"\beta"',distance=2em]
\end{tikzcd}
\]

$A$ is $\tau$-tilting finite, and we can explicitly compute its mutation quiver $Q(s\tau\text{-tilt} A)$ as shown in Figure \ref{mutation-quiver-example-decomposing-modules}. Since $\Hom_A(P(2),P(1))$ is $2$-dimensional we have by Lemma \ref{inherit-tech} a 2-term pre-silting complex with $g$-vector $(i,-i + 1)$ for any $i \geq 0$. Since there are only finitely many $\tau$-rigid modules over $A$, almost all of these given complexes must decompose. 

\begin{figure}[h]
	\[\begin{tikzcd}
		& {(P(1)\oplus P(2),0)} \\
		{(P(2),P(1))} && {(P(1) \oplus T_1,0)} \\
		&& {(S(1)\oplus T_1,0)} \\
		&& {(S(1),P(2))} \\
		& {(0,P(1)\oplus P(2))}
		\arrow[from=1-2, to=2-1]
		\arrow[from=1-2, to=2-3]
		\arrow[from=2-3, to=3-3]
		\arrow[from=3-3, to=4-3]
		\arrow[from=4-3, to=5-2]
		\arrow[from=2-1, to=5-2]
	\end{tikzcd}\]
	\caption{The mutation quiver of the algebra in example \ref{example-decomposing-modules}, where $T_1$ is the $\tau$-rigid module with $g$-vector $(2,-1)$}
	\label{mutation-quiver-example-decomposing-modules}
\end{figure}
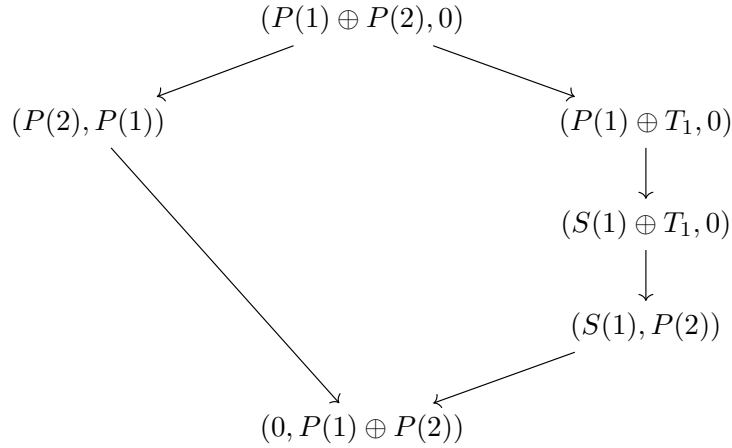

\end{example}

The following lemma is then of interest.

\begin{lemma}\label{indecomposability-condition}
	Let $\Lambda = K_2$ as in Example $\ref{K2example}$. Let $\Gamma$ be an algebra with $n$ simple modules and ordered projectives $Q(i)$, where $i \in [1,\dots,n]$, and a pair of projectives $(Q(2),Q(1))$ such that $hom_\Gamma(Q(2),Q(1)) = 2$. Let the projectives of $\Gamma$ be ordered such that $g^{Q_1} = (1,0,0,\dots,0)$ and $g^{Q_2} = (0,1,0,\dots,0)$. Let $F$ be the functor as in Example \ref{K2example}.

	We consider the class $\mathcal{C}$ of 2-term pre-silting complexes over $K_2$ coming from projective presentations of $\tau$-rigid indecomposable modules on the form \[P(2)' \to P(1)' \] where $P(i)'$ lies in $\add P(i)$.

	 Assume that $(1,-1,0,0,\dots,0)$ is not a $g$-vector of $\Gamma$. Then all complexes in $\mathcal{C}$ are transported to indecomposable 2-term pre-silting complexes over $\Gamma$ via $F$ as defined in Example $\ref{K2example}$.
\end{lemma} 

\begin{proof}
	We are guaranteed a sequence of (not necessarily basic nor indecomposable) 2-term pre-silting complexes $\{\mathbb{P}_i\}_{i \in \mathbb{N}}$ with $g$-vectors on the form $(i+1,-i,0,0,\dots,0)$.
	
	$\mathbb{P}_0$ is concentrated in degree $0$ where its homology is a projective indecomposable module. We proceed by contradiction. Assume that $\mathbb{P}_j$ decomposes and that $\mathbb{P}_k$ is indecomposable for all $k < j$. We know by \ref{inherit-tech} that $\mathbb{P}_{j-1} \oplus \mathbb{P}_j$ is 2-term pre-silting. If $\mathbb{P}_{j-1}$ is a summand of $\mathbb{P}_j$, then $\mathbb{P}_j = \mathbb{P}_{j-1} \oplus \mathbb{X}$, where $\mathbb{X}$ must have $g$-vector $(1,-1,0,0,\dots,0)$, against our assumptions.
	
	Since $\mathbb{P}_{j-1}$ is not a summand of $\mathbb{P}_j$, we know that $\mathbb{P}_{j-1} \oplus \mathbb{P}_j$ has at least three pairwise non-isomorphic indecomposable rigid summands, say $\mathbb{X}_1,\mathbb{X}_2$ and $\mathbb{X}_3$. They must all have indices on the form $(x,y,0,0,\dots,0)$, thus the span of $(g^{\mathbb{X}_1},g^{\mathbb{X}_2},g^{\mathbb{X}_3})$ in $\mathbb{Z}^n$ is at most two-dimensional. As $G$-matrices are invertible, we can now derive a contradiction. In particular, we have that any $G$-matrix $G^T$ where $T$ is a support $\tau$-tilting pair is an invertible integer matrix. If we let $T$ be for example a support $\tau$ tilting pair whose corresponding $2$-term silting object has $\mathbb{X}_1 \oplus \mathbb{X}_2 \oplus \mathbb{X}_3$ as a summand, then $G^T$ has rank at most $n-1$, a contradiction.

\end{proof}


\subsection{Some algebras with at most two components in their mutation quivers}

We begin by stating the main theorem of this section. 

\begin{theorem}\label{k2-reduction}
	Let $A$ be the algebra $kQ/r^3$ where $Q$ is the quiver	
	
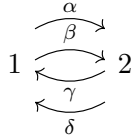
\[\begin{tikzcd}
		1 
		\arrow[r,bend left,shift right = 0.5ex, "\beta"]
		\arrow[r,bend left, shift left = 2ex, "\alpha"]
		& 2 \arrow[l,bend left,shift right = 0.5ex,"\gamma"]
		\arrow[l,bend left,shift left = 2ex, "\delta"]  \\
	\end{tikzcd}
	\]
	
	Then $Q(s\tau\text{-tilt } A)$ has exactly two components.
	
\end{theorem}

The following corollary shows that Theorem \ref{k2-reduction} allows us to understand the $\tau$-tilting theory of quotient algebras of $A$. In particular, we immediately obtain a generalization of \cite[Theorem 6.17]{dij17}, where a specific quotient of $A$ is studied. Indeed, both the statement and proof of \ref{k2-reduction} is inspired by aforementioned result of Demonet, Iyama and Jasso.

\begin{corollary}\label{k2-reduction-quotient}
	Let $A$ be as in Theorem \ref{k2-reduction}, and let $I$ be an ideal contained in the radical of $A$. Then $Q(s\tau\text{-tilt } A/I)$ has at most two components.
	
\end{corollary}
 
\begin{proof}
	Walls of $B = A/I$ are by \cite[Lemma 4.13]{Br_stle_2019} also walls in $A$. Let $T$ be an arbitrary support $\tau$-tilting pair over $A/I$. We show directly that it lies in the same component of the mutation quiver as $B = (A/I,0)$ or $B^\dagger = (0,A/I)$. Note that $C(T)$ need not be a chamber in the wall and chamber structure of $A$, but every point $x \in C(T)$ is contained in a chamber in the wall and chamber structure of $A$. Fix such a point $x$, and assume that it is contained in $C(T')$. Of course, $T'$ lies in the component of either $(A,0)$ or $(0,A)$, and there is therefore a $\mathcal{D}$-generic path $\gamma$ from one of the chambers $C(A,0)$ or $C(0,A)$ to $C(T')$ passing only finitely many walls. It is not hard to see that $\gamma$ is $\mathcal{D}$-generic as a path in the wall and chamber structure of $B$, and passes only finitely many walls here. Thus $T$ must lie in either the same component as $B$ or $B^\dagger$ in $Q(s\tau\text{-tilt } B)$.
\end{proof}

Before proving Theorem \ref{k2-reduction}, we need a small lemma.

\begin{lemma}\label{condition-gvector-k2reduction}
	The algebra $A$ as in Theorem \ref{k2-reduction} has no $g$-vectors on the form $(1,-1)$ or $(-1,1)$
\end{lemma}

\begin{proof}
	Let $B = A/(\gamma,\delta)$. Then $B$ is simply the Kronecker algebra, and thus have a sequence of $g$-vectors $(2,-1),(3,-2),(4,-3),\dots$ which by \cite[Lemma 4.13]{Br_stle_2019} must be walls in $A$. Thus if $(1,-1)$ was a $g$-vector of $A$ it would be the wall of a cone $C(T)$ for some support $\tau$-tilting pair $T$. This pair $T$ has a unique right mutation, say $T'$. The cone $C(T')$ would share a wall with $C(T)$ and lie counterclockwise to the cone of $T$, and thus would necessarily contain infinitely many walls. This is impossible as the interior of cones are chambers, and hence cannot contain walls by definition.
	A symmetrical argument shows that $(-1,1)$ cannot be a $g$-vector of $A$.
\end{proof}


\begin{proof}[Proof of Theorem \ref{k2-reduction}]
	We first investigate the component of $Q(s\tau \text{-tilt } A)$ containing $(P(1) \oplus P(2),0)$.
	
	Let $K$ be the Kronecker algebra, with projectives $P_K(1)$ and $P_K(2)$. Note that $hom_A(P(1),P(2)) = 2 = hom_A(P(2),P(1))$, as paths of length $3$ are killed in $A$. This symmetry allows us to utilize Lemma \ref{inherit-tech} in two ways.
	
	We define the functor $F$ by sending $P_K(1)$ to $P(1)$ and $P_K(2)$ to $P(2)$ and respecting morphisms as in Example \ref{K2example}.
	
	By Lemma \ref{condition-gvector-k2reduction}, we may utilize Lemma \ref{indecomposability-condition}. This means that the $g$-vectors on the form $(i+1,-i)$ for $i \geq 0$ guaranteed by Lemma \ref{inherit-tech} correspond to indecomposable modules. If we denote by $T_i$ the $\tau$-rigid module with $g$-vector $(i+1,-i)$, we know that $T_i \oplus T_{i+1}$ is $\tau$-rigid, meaning that there cannot be any $g$-vectors not on the form $(i+1,-i)$ lying in the union of the cones $\cup_{i \geq 0} C(T_i \oplus T_{i+1})$. Notice that the closure of these cones is the positive span of the vectors $(1,0)$ and $(1,-1)$. 
	
	Now, we can apply two symmetries to complete the argument. First, we may utilize Lemma $\ref{inherit-tech}$ on the functor sending $P(1)$ to $P_K(2)$ and $P(2)$ to $P_K(1)$, giving $g$-vectors on the form $(-i,i+1)$. These $g$-vectors show that the union of the cones of the support $\tau$-tilting pairs in the same component as $(A,0)$ hits all points $(x,y)$ in $\mathbb{R}^2$ where $x + y > 0$.
	
	Secondly, we may utilize the general duality in $\tau$-tilting theory, which allows us to compute right mutations using left mutations in the opposite algebra. Since $A$ is symmetric, and $g$-vectors respect this duality, we can conclude that the closure of all the cones we have identified is in fact $\mathbb{R}^2$, leaving no room for more chambers. Thus we have identified all support $\tau$-tilting objects of $A$ up to isomorphism.

\end{proof}


\subsection{Mutation quivers with multiple connected components}

We now give an example of an algebra with more than one component in its mutation quiver.

\begin{lemma}\label{4components}
	Let $A = kQ/r^2$ with the $Q$ as displayed below.
	
	
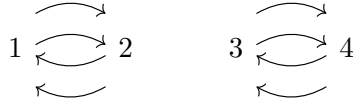
\[\begin{tikzcd}
		1 
		\arrow[r,bend left,shift right = 0.5ex, ""]
		\arrow[r,bend left, shift left = 2ex, ""]
		& 2 \arrow[l,bend left,shift right = 0.5ex,""]
		\arrow[l,bend left,shift left = 2ex, ""]  & 
		3 
		\arrow[r,bend left,shift right = 0.5ex, ""]
		\arrow[r,bend left, shift left = 2ex, ""]
		& 4 \arrow[l,bend left,shift right = 0.5ex,""]
		\arrow[l,bend left,shift left = 2ex, ""]  \\
	\end{tikzcd}
	\]
	

	Then $Q(s\tau\text{-tilt} A)$ has exactly $4$ connected components.
\end{lemma}

\begin{proof}
	This follows immediately from Corollary \ref{k2-reduction-quotient} and Lemma \ref{disconnected_mutation_quiver}.
\end{proof}

Observe that the above result can be generalized to show that there is no bound on the number of components a mutation quiver can have, as we can simply take arbitrarily many copies of the algebra with two simple modules that we started with. We are not aware of any algebras with infinitely many components in its mutation quiver and consider it an interesting question whether such an algebra exists. 

The theorem below shows that we can also construct a connected algebra where the property of having more than two components in its mutation quiver is preserved.

\begin{theorem}\label{connected_more_components}
	Let $A = kQ/r^2$ with the quiver $Q$ as displayed below.
	
			
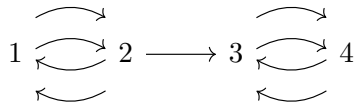
\[\begin{tikzcd}
			1 
			\arrow[r,bend left,shift right = 0.5ex, ""]
			\arrow[r,bend left, shift left = 2ex, ""]
			& 2 \arrow[l,bend left,shift right = 0.5ex,""]
			\arrow[l,bend left,shift left = 2ex, ""]  
				\arrow[r, ""] & 
			3 
			\arrow[r,bend left,shift right = 0.5ex, ""]
			\arrow[r,bend left, shift left = 2ex, ""]
			& 4 \arrow[l,bend left,shift right = 0.5ex,""]
			\arrow[l,bend left,shift left = 2ex, ""]  \\
		\end{tikzcd}
		\]
	

	Then $Q(s\tau\text{-tilt} A)$ has at least $4$ connected components.
\end{theorem}

\begin{proof}
	
	Note that $\Hom_A(P(1)\oplus P(2),P(3)\oplus P(4)) = 0$ as there are no paths from vertices $3,4$ to vertices $1,2$. Thus \[T_1 = (P(3) \oplus P(4), P(1) \oplus P(2))\] is a support $\tau$-tilting pair. We intend now to show that it may not lie within a finite number of right or left mutations from $(A,0)$ and $(0,A)$. This means that $T_1$ must lie in a third component of $Q(s\tau \text{-tilt} A)$, as wanted.
	
	From \cite[Lemma 4.13]{Br_stle_2019} walls in the algebra in Lemma \ref{4components} are also walls in the wall and chamber structure of $A$. Let $B$ be the algebra in Lemma \ref{4components}.
	
	Consider the support $\tau$-tilting pair $T_{B,1} = (P_B(3) \oplus P_B(4), P_B(1) \oplus P_B(2))$ over $B$. Then $T_{B,1}$ and $T_1$ will have equal $G$-matrices.
	
	\[G^{T_1} = \begin{bmatrix}
		0 & 0 & -1 & 0 \\ 0 & 0 & 0 & -1 \\ 1 & 0 & 0 & 0 \\ 0 & 1 & 0 & 0
	\end{bmatrix} = G^{T_{B,1}}\]

	This gives that $g^{T_1}$ lies in the cone of $T_{B,1}$. Assume now that there is a $\mathcal{D}$-generic path $\gamma$ from $g^{T_1}$ to $g^{(A,0)} = (1,1,1,1)$ in $Q(s\tau\text{-tilt } A)$. In the wall and chamber structure of $B$, this path is also $\mathcal{D}$-generic and must cross infinitely may walls since $(B,0)$ and $T_{B,1}$ lie in different components of $Q(s\tau\text{-tilt} B)$. Now, since walls in $B$ induce walls in $A$, we see that $\gamma$ must pass infinitely many walls also in $A$. 
	
	Then, by the properties of $\mathcal{D}$-generic paths $T_1$, can not lie in the same component as $(A,0)$. A similar argument shows that it does not lie in the same component as $(0,A)$.
	
	To identify a fourth component, observe first that $(P(1)\oplus P(2),P(3)\oplus P(4))$ is not support $\tau$-tilting, as there is a nonzero map $P(3) \to P(1)$ induced by the path of length $1$ from point $1$ to point $3$. We consider instead the mutation $\mu_{P(1)}(P(1)\oplus P(2)\oplus P(3)\oplus P(4))$.
	
	Since $\Hom(P(1),P(4)) = 0 = \Hom_A(P(1),P(3))$, the mutation $\mu_{P(1)}(P(1)\oplus P(2)\oplus P(3)\oplus P(4))$ may be computed as a co-kernel of $P(2) \oplus P(2)$ using the mutation technique in $\tau$-tilting theory as developed in \cite{tau}.
	
	Let then \[X \oplus P(2) \oplus P(3) \oplus P(4) = \mu_{P(1)}(P(1)\oplus P(2) \oplus P(3)\oplus P(4))\]  $X$ can easily be computed to be a module with dimension vector $(3,2,0,0)$ and $g$-vector $(-1,2,0,0)$. 
	
	Since $\Hom_A(P(3)\oplus P(4),X\oplus P(2)) = 0$ we get that \[T_2 = (X \oplus P(2), P(3) \oplus P(4))\] is a support $\tau$-tilting pair. Consider $T_{B,2} = \mu_{P_B(1)}((P_B(1) \oplus P_B(2), P_B(3) \oplus P_B(4)))$. Then $T_{B,2}$ and $T_2$ will have equal $G$-matrices.
	
	\[G^{T_2} = \begin{bmatrix}
		-1 & 0 & 0 & 0 \\ 2 & 1 & 0 & 0 \\ 0 & 0 & -1 & 0 \\ 0 & 0 & 0 & -1
	\end{bmatrix} = G^{T_{B,2}}\]
	
	This means in particular that $g^{T_2}$ lies in the cone of $T_{B,2}$. We can then argue as above that $T_2$ can not lie in the component containing $(A,0)$ or the component containing $(0,A)$. 
	
	Lastly, $T_2$ and $T_1$ themselves must also lie in two distinct components of $Q(s\tau\text{-tilt } A)$, because their $g$-vectors lie in cones corresponding to support $\tau$-tilting objects of $B$ lying in distinct components.
	
	This concludes our proof.
\end{proof}

\printbibliography

\end{document}